\documentclass[10pt]{amsart}
\usepackage[latin1]{inputenc}
\usepackage{epsfig}
\usepackage{color}
\usepackage[british,english]{babel}
\usepackage{amsthm}
\usepackage{amsmath}
\usepackage{amsfonts}
\usepackage{amssymb}
\usepackage{graphicx}
\newtheorem{corollary}{Corollary}[section]

\newtheorem{Definition}[corollary]{Definition}

\newtheorem{lemma}[corollary]{Lemma}
\newtheorem{prp}[corollary]{Proposition}
\newtheorem{remark}[corollary]{Remark}
\newtheorem{thm}[corollary]{Theorem}

\newtheorem{asmp}[corollary]{Assumption}
\newfont{\sBlackboard}{msbm10 scaled 900}

\newcommand{\mylabel}[1]{\label{#1}
            \ifx\undefined\stillediting
            \else \fbox{$#1$}\fi }
\newcommand{\BE}{\begin{equation}}

\newcommand{\EEQ}{\end{equation}}
\newcommand{\rfb}[1]{\mbox{\rm
   (\ref{#1})}\ifx\undefined\stillediting\else:\fbox{$#1$}\fi}

\newfont{\Blackboard}{msbm10 scaled 1200}

\newfont{\roma}{cmr10 scaled 1200}

\def\CC{\rm \hbox{C\kern-.56em\raise.4ex
         \hbox{$\scriptscriptstyle |$}\kern+0.5 em }}

\newcommand{\mm}    {{\hbox{\hskip 0.5pt}}}

\newcommand{\bluff} {{\hbox{\raise 15pt \hbox{\mm}}}}
\begin{document}
\thispagestyle{empty}
\title[Uniform stabilization for the semi-linear wave equation]{Uniform stabilization for the semi-linear wave equation with nonlinear Kelvin-Voigt damping}
\author{Ka\"{\i}s Ammari}
\address{LR Analysis and Control of PDEs, LR 22ES03, Department of Mathematics, Faculty of Sciences of Monastir, University of Monastir, Tunisia}
\email{kais.ammari@fsm.rnu.tn}

\author{Marcelo M. Cavalcanti}
\address{Department of Mathematics, Universidade Estadual de Maring\`a,
Maring\`a, PR87020-900, Brazil}
\email{mmcavalcanti@uem.br}

\author{Sabeur Mansouri}
\address{LR Analysis and Control of PDEs, LR 22ES03, Department of Mathematics, Faculty of Sciences of Monastir, University of Monastir, Tunisia}
\email{m.sabeur1@gmail.com} 

\begin{abstract} 
This paper is concerned with the decay estimate of solutions to the semilinear wave equation subject to two localized dampings in a bounded domain. The first one is of the nonlinear Kelvin-Voigt type and is distributed around a neighborhood of the boundary according to the Geometric Control Condition. While the second one is a frictional damping and we consider it hurting the geometric condition of control. We show uniform decay rate results of the corresponding energy for all initial data taken in bounded sets of finite energy phase-space. The proof is based on obtaining an observability inequality which combines unique continuation properties and the tools of the Microlocal Analysis Theory.

\end{abstract}


\maketitle

\tableofcontents

\makeatletter
\def\section{\@startsection {section}{1}{\z@}{-3.5ex plus -1ex minus
    -.2ex}{2.3ex plus .2ex}{\large\bf}}
\makeatother
%
\def\be{\begin{equation}}
\def\ee{\end{equation}}
 \def\pd{\partial}
\def\ds{\displaystyle}
 
\section{Introduction}
\subsection{Description of the problem}
In this paper, we study the analysis of the uniform decay rates of solutions to the wave equation subject to a non-linear Kelvin-Voigt damping. We consider the following problem
\begin{equation}\label{1}
 \left\{\begin{array}{lll}
       \partial_t^2u-\Delta u-\mathop{\rm div} \big( a(x)g(\nabla\partial_tu)\big)+\eta(x)\partial_tu+f(u)=0,\;\;&\mbox{in}& \;\; \Omega \times (0,\infty),\\
			
                    u =0, &\mbox{on}&\partial\Omega\times (0,\infty),\\
                       
 \displaystyle  u(x,0)=u_{0}(x),\;\;\;\partial_tu(x,0)=u_{1}(x), \;\;&\mbox{in}& \;\; \Omega,
                   
 \end{array}\right.
 \end{equation} 
where $\Omega$ is a bounded domain of $\mathbb{R}^n$, $n \geq 1$, with a smooth boundary $\partial \Omega$. We denote by $\omega$ the intersection of the domain $\Omega$ with a neighborhood of $\partial \Omega$ in $\mathbb{R}^n$ and we call it a neighborhood of the boundary of $\Omega$. Therewith, $f$ and $g$ are real valued functions satisfying Assumption \ref{a1. 1} and Assumption \ref{a1.2} below.
\begin{asmp}\label{a1. 1} $f:\mathbb{R}\rightarrow \mathbb{R}$ is a $C^2(\mathbb{R})$ function which is assumed to grow sub-critically satisfying the sign condition $f(s)s \geq 0$, $\forall s \in \mathbb{R}$, and, furthermore:
\begin{equation}\label{2}
f(0)=0,\;\;\; |f^j(s)|\leq k_0(1+|s|)^{p-j},\;\; j=1,2,\;\; \forall s\in \mathbb{R},
\end{equation}
which implies, in particular, that for some $C > 0$,
\begin{equation}\label{3}
|f(s_1)-f(s_2)|\leq C(1+|s_1|^{p-1}+|s_2|^{p-1})|s_1-s_2|,\;\;\forall s_1,\;s_2\in \mathbb{R},
\end{equation}
with
\begin{equation}\label{4}
1\leq p<\frac{n}{n-2} \; \mbox{ if } \; n\geq 3 \;\mbox{ or }\; p\geq 1\;\mbox{ if }\; n=1,2.
\end{equation}
In addition, 
\begin{equation}\label{5}
0\leq F(s)\leq f(s)s,\;\;\mbox{ for all }\;\; s\in \mathbb{R},
\end{equation}
where $F(s):=\int_0^sf(r)\;dr$.
\end{asmp}
\begin{asmp}\label{a1.2}
The feedback function 
$$ g:\mathbb{R}^n\longrightarrow \mathbb{R}^n $$
$$ s\longmapsto [g_i(s_i)]_{i=1,...,n}, $$
where, for all $i=1,...,n$, the component $g_i:\mathbb{R}\longrightarrow \mathbb{R}$ is a function such that 
\begin{equation}\label{6}
 \left\{\begin{array}{lll}
g_i \mbox{ is continuous, monotone increasing on } \mathbb{R} \mbox{ and } g_i(0)=0,\\
g_i(s_i)s_i>0,\;\;\forall s\neq 0,\\
ms_i^2\leq g_i(s_i)s_i\leq M s_i^2 \mbox{ for } |s_i|\geq 1, \mbox{ for some } m,\;M; 0<m\leq M.
 \end{array}\right.
\end{equation}
\end{asmp}
The functions $a(\cdot)$ and $\eta(\cdot)$ responsible for the localized dissipative effects are smooth functions satisfying:
\begin{asmp}\label{a1.3}
$a(\cdot) \in L^\infty(\Omega) $ is a non-negative function. Further, there exists a compact, connected subset $A\subset \Omega$ with smooth boundary and no-empty interior such that
$$ A:=\left\{x\in \Omega:\; a(x)=0\right\}. $$
We also suppose that $a(\cdot)\in C^0(\overline{\omega})$, where $\omega:=\Omega/A.$
\end{asmp}
\begin{asmp}\label{a1.6}
The frictional damping $\eta\in C^0(\overline{\Omega})$ is a nonnegative function and is effective  
in a neighbourhood of the boundary of the set $A=\left\{x\in \Omega:\; a(x)=0\right\}$, that is, 
$$ \eta(x) \geq \eta_0 >0$$
in $\mathcal{O}_\varepsilon=\left\{x\in \Omega:\; d(x,y)<\varepsilon,\;y\in \partial A\right\}$ for $\varepsilon>0$ small enough. 
\end{asmp}
Let us assume that the following assumptions are also made:
\begin{asmp}\label{a1.4}
$\omega$ is a subset of $\Omega$ which controls geometrically $\Omega$, i.e. there exists $T_0>0$ such that every ray of the geometric
optics enters $\omega$ in a time $t<T_0$.
\end{asmp}
\begin{asmp}\label{a1.5}
For every $T>0$, the unique solution 
$$u\in C(0,T;L^2(\Omega))\cap C^1(0,T;H^{-1}(\Omega)), $$
that satisfies
\begin{equation}
 \left\{\begin{array}{lll}
       \partial_t^2u-\Delta u+V(x,t)u=0\;\;&\mbox{in}& \;\; \Omega \times (0,\infty),\\
                    u =0 \;\;&\mbox{in}& \;\; \omega\times (0,\infty),

 \end{array}\right.
 \end{equation}
where $V(x,t) \in L^\infty \left(0,T;L^\infty(\Omega)\right)$, is the trivial one.
\end{asmp}
\subsection{Previous results and methodology}
The stabilization of the wave equation have been extensively investigated in literature by many authors. Some of them considered 
linear or nonlinear viscous damping of the form $a u_t$ or $ag(u_t)$ in different setting. To name few, we mention the following papers \cite{LT9,LT10,LT11,LT12,LT13,LT14,LT15,LT16,LT17,LT18,LT19,LT20,LT21} and many references therein. In \cite{LT9}, E. Zuazua proved that the energy of the linear wave equation decays exponentially if the damping region, where $a(x)\geq a_0>0$, contains a neighbourhood of $\partial \Omega$ or at least a neighbourhood of $\Gamma(x_0)=\left\{x\in \partial \Omega,\;(x-x_0)\cdot \nu(x)\geq 0\right\}$ where $\nu$ is the outward unit normal to $\Omega$ and $x_0\in \mathbb{R}^n$. 

\medskip

Thereafter, Bardos et al. in \cite{LT6} established, using microlocal analysis, sharp sufficient conditions for the observation, control and stabilization of the same linear equation on a compact Riemannian manifold $(M, g)$ with boundary.

\medskip

In \cite{LT23}, using the piecewise multiplier method, P. Martinez weakened the usual geometrical condition on the localization of the damping and eliminated the polynomial growth of the feedback to prove the stability result with a precise decay rate estimates. Later, Cavalcanti et al. in \cite{LT24,LT13,LT14} studied the wave equation where a localized nonlinear damping is considered. In the first one, the authors consider compact surfaces, while the other two papers deal with compact manifolds. Thanks to a damping effective in the complement of a set satisfying a technical condition connected to the existence of multipliers, they were able to show uniform and optimal decay rates of the energy in \cite{LT13}, whereas in \cite{LT14}, the optimal decay rate was proved under arbitrary small control regions based on the construction of a special multiplier. In \cite{LT25}, Alabau-Boussouira considered a family of nonlinear feedbacks, varying from very weak nonlinear dissipation to polynomial or polynomial-logarithmic decaying feedbacks at the origin, for showing sharp or quasi-optimal upper energy decay rates. Moreover, Alabau-Boussouira and Ammari obtained in \cite{LT26} sharp-energy decay rates for the nonlinearly damped abstract infinite-dimensional systems,
combining optimal geometric conditions as provided by Bardos et al \cite{LT6} with an optimal weight convexity method of Alabau-Boussouira (see \cite{LT25} and \cite{LT27}). This approach has been later extended to uniform discretization of nonlinearly damped evolution equations, see \cite{LT11}.

\medskip

Furthermore, there are a numerous papers concerning the wave equation with the linear viscoelastic damping of Kelvin-Voigt type in a bounded domain. In \cite{LT28}, Liu an Liu investigated an one-dimensional linear wave equation with the Kelvin-Voigt damping is effective in a subinterval which models an elastic string with one segment made of viscoelastic material and the other of elastic material. Since the discontinuity in the damping coefficient and the high order of damping operator, the authors showed that the energy of that system does not decay exponentially. This was partially confirmed in \cite{LT30} where the uniform stability of an one-dimensional wave equation is obtained  for local viscoelastic damping with differentiable coefficients, see also \cite{KLS}. In the same spirit, a result of uniform stability is established in \cite{LT31} in the higher dimensional case which present a more delicate problem due to the unboundedness of viscoelastic damping and the lack of regularity of solution. To overtake this difficulties the authors considered a special case where the viscoelastic damping is localized in a neighborhood of the whole of the boundary under some conditions. It also worth mentioning that this geometrical condition has been treated in \cite{LT34,KH} by considering a singular Kelvin-Voigt damping which is localized far away from the boundary. In order to deal with such a situation, Tebou has proposed a constructive frequency approach to improve this result. He has weakened the hypothesis on the damping coefficients and especially on the feedback control region, see \cite{LT32}. More recently, Tebou et al. studied in \cite{LT16} the wave equation with two types of linear locally distributed damping mechanisms: a frictional damping and a Kelvin-Voigt type
damping.  Using a combination of the multiplier techniques and the frequency domain approach, they established that a convenient interaction of
the two damping mechanisms is powerful enough for the exponential stability of the dynamical system, provided that the coefficient of the Kelvin-Voigt damping is smooth enough and satisfies a structural condition. This latter result has been generalized in \cite{LT35} to the semilinear case in an inhomogeneous medium subject to a Kelvin-Voigt type damping acting in a neighborhood of the boundary or in a mesh totally distributed in the domain with measure arbitrarily small. 

\medskip

Motivated by all these works, the main concern of this paper is to consider a wave equation, in the higher dimensional case, subject to a nonlinear locally distributed Kelvin-Voigt damping to which we associate a linear frictional one. What makes this problem interesting is the nature of the operator defining the damping. Firstly, this type of damping generates a nonlinear and unbounded operator. It is also of great interest to point that the domain consists of two different materials, that is, there is an interaction between the elastic component (the portion of domain where $a\equiv 0$) and the nonlinear viscoelastic component (the portion of domain where $a> 0$). Further, the solution of the corresponding dynamical system does not have the requisite smoothness to apply the usual multiplier method. All these points makes the associated stabilization problem more difficult to study, especially in the multidimensional setting. In turn, to the authors' best knowledge, the wave equation subject to nonlinear locally distributed Kelvin-Voigt damping was not consider in literature (and a fortiori in the above references and the references therein). Finally, it is important to remark that the problem under consideration generalize many works in the literature, in particular, the papers mentioned above.

\medskip

To obtain the desired result, our strategy is based on microlocal analysis techniques combined with unique continuation properties to establish an observability inequality. In the following, we will give a briefly explication for this strategy. 
The energy associated to problem \eqref{1} is given by
\begin{equation}
\label{7}
E_u(t):=\frac{1}{2}\int_\Omega\left\{ |\partial_tu(x,t)|^2+|\nabla u(x,t)|^2\right\}\;dx+\int_\Omega F(u(x,t))\;dx,
\end{equation}
which satisfies the following identity 
$$E_u(t)+\sum_{i=1}^n\int_{0}^{t}\int_\Omega a(x)g_i(\partial_{x_i} \partial_tu(x,s))\partial_{x_i} \partial_tu(x,s)\;dxds
$$
\begin{equation}\label{771}
+\int_{0}^{t}\int_\Omega\eta(x)|\partial_tu(x,s)|^2\;dxds=E_u(0),\;\;\mbox{ for all } t\in [0,+\infty).
\end{equation}
The main objective of the present paper is to prove the existence and uniqueness for weak solutions to problem (1.1) and, in addition, that those solutions decay uniformly to zero, that is, 
$$E_u(t)\leq S\left(\frac{t}{T}-1\right),\;\; \mbox{ for all } t>T_0,$$
where the contraction semigroup $S(t)$ is the solution of the differential equation 
$$ \frac{d}{dt}S(t)+\mathcal{R}(S(t))=0,\;\;\;S(0)=E(0), $$
where $\mathcal{R}$ is given in \eqref{17'''}, for all mild solutions to problem \eqref{1} provided that the initial data are taken in bounded sets of the corresponding energy space. This result is a local stabilization result. Inspired by Dehman, G\'erard, Lebeau \cite{LT36} or Dehman, Lebeau and Zuazua \cite{LT18}, we give a direct proof of the inverse inequality to problem \eqref{1} namely, we prove that there exist $T>0$ and a positive constant $C=C(T)$ such that
\begin{equation}\label{772}
E_u(0)\leq C\left( \int_0^T\int_\Omega a(x)\left(|\nabla \partial_t u|^2 +|g(\nabla \partial_t u)|^2 \right)dxdt+ \int_0^T\int_\Omega\eta(x)|\partial_t u|^2\;dxdt \right),
\end{equation}
provided the initial data are taken in bounded sets of the corresponding energy space and the nonlinarity $f$ is taken as in the Assumption \ref{a1. 1}. 

\medskip

Indeed, for a given $\delta>0$ we denote by $\mathcal{O}_\delta$ the neighbourhood of $\partial A$ given by 
$$ \mathcal{O}_\delta=\left\{ x\in \Omega:\;d(x,y)<\delta,\;y\in \partial A \right\}. $$
First of all, it is worth mentioning that the basic damping to prove the desired uniform result is the nonlinear viscoelstic one. Indeed, note that if $a(\cdot)$ vanish in the whole of the domain $\Omega$, the frictional damping $\eta(\cdot)$ is not strong enough to provide the exponential and uniform decay of the energy since in this case the geometric control condition GCC is violated, for instance see \cite{LT6} or also \cite{LT12}. However, the damping $\eta(\cdot)$ plays an important role since the lack of convergence due to the Kelvin-voigt damping which will be explained next.
To prove \eqref{772} and therefore the stability result, we will argue by contradiction and we will obtain a sequence $\left\{v_k\right\}_{k\in\mathbb{N}}$ of solutions to normalized problem associated to \eqref{1} such that $E_{v_k}(0)=1$. To do so, we have to consider $\mathcal{V}=\left\{ x\in A:\; d(x,y)>\varepsilon/2,\;y\in \partial A \right\}$. In order to obtain a contradiction we need to prove that $E_{v_k}(0)\rightarrow 0$ as $k\rightarrow +\infty$. We shall show it, by exploiting the properties of the function $a(\cdot)$, $\eta(\cdot)$, $f$ and $F$,  and a unique continuation principle, that 
\begin{equation}\label{773}
\int_0^T\int_{\Omega\backslash\mathcal{V}} | \partial_t v_k|^2 \;dxdt\rightarrow 0 ,\; \mbox{ as } m \rightarrow +\infty.
\end{equation}
Next we are going to propagate this convergence of $\partial_tv_k$ to $0$ from $L^2((\Omega/\mathcal{V}) \times (0,T))$ to the whole space $L^2(\Omega \times (0,T))$. For that, we will use microlocal analysis arguments. Indeed, let us consider $\chi$ be the microlocal defect measure associated with $\left\{\partial_tv_k\right\}_k$ in $L^2_{loc}(\Omega\times(0,T))$. We are going to establish the convergence 
\begin{equation}\label{774}
D(\partial_tv_k):=(\partial^2_t-\Delta)\partial_tv_k \rightarrow 0, \mbox{ strongly in } H^{-2}_{loc}(\Omega\times(0,T)) \mbox{ as } k\rightarrow \infty.
\end{equation}
In view of the convergence \eqref{774}, we can deduce that the $\chi$ is contained in the characteristic set of the wave equation. But, this convergence is not enough to obtain the desired propagation. Indeed, we need the stronger convergence 
\begin{equation}\label{775}
(\partial^2_t-\Delta)\partial_tv_k \rightarrow 0, \mbox{ strongly in } H^{-1}_{loc}(\Omega\times(0,T)) \mbox{ as } k\rightarrow \infty.
\end{equation}
This problem of regularity due to the following weak convergence:
\begin{equation}
\partial_t\left(\mathop{\rm div} \left[ \frac{1}{\sigma_k}a(x)g(\nabla\partial_tu_k)\right]\right) \rightarrow 0, \mbox{ strongly in } H^{-2}_{loc}(\Omega\times(0,T)) \mbox{ as } k\rightarrow \infty,
\end{equation}
where $\sigma_k = \sqrt{E_{u_k}(0)}, u_k = \sigma_k \, v_k.$

Thus, the presence of the frictional damping in the neighborhood $\mathcal{O}_\epsilon$ of the boundary of $\partial A $ will play a basic role. In fact, we have $\mathop{\rm div} \left( \frac{1}{\sigma_k}a(x)g(\nabla\partial_tu_k)\right)=0$ in $A\times (0,T)$ and, thus, we find that
\begin{equation}
D\partial_tv_k=\partial_t\left(-\eta(x)\partial_tv_k-\frac{1}{\sigma_k}f(u_k)\right)\rightarrow 0\mbox{ in } H^{-1}_{loc}(\mbox{int}(A)\times(0,T)) \mbox{ as } k\rightarrow \infty,
\end{equation}
which implies that $\chi$ propagates along the bicharacteristic flow of the operator $D(\cdot)$. Particularly, if there is $\mu_0=(t_0,x_0,\tau_0,\xi_0)$ does not belong to the supp$(\chi)$, the whole bicharacteristic issued from $\mu_0$ is out of supp$(\chi)$.
Further, since supp$\chi\subset \mathcal{V}\times (0,T)$, and the frictional damping effects in both sides of the boundary $\partial A$, we can propagate the kinetic energy from $(\mathcal{O}_{\epsilon/2}\cap A)\times (0,T)$ towards the set $\mathcal{V}\times (0,T)$.

\medskip

Then, we find supp$(\chi)=\emptyset$ and $\partial_tv_k\rightarrow 0$ in $L^2_{loc}(\Omega\times(0,T))$. Therewith, since some computations, we can obtain 
\begin{equation}
\int_0^T\int_\Omega|\partial_tv_k(x,t)|^2\;dxdt\rightarrow 0,\mbox{ as } k\rightarrow \infty.
\end{equation}
Finally, this convergence combined with an argument of equipartition of energy, we can conclude the desired contradiction $E_{v_k}(0)\rightarrow 0$.

\medskip 

The rest of the paper is organized as follows: section 2 is devoted to the well-posedness of the problem \eqref{1}. In section 3, we will study the uniform stability of our problem subject to strong damping. We prove the nonlinear observability inequality associated to problem  \eqref{1} by considering refined arguments of microlocal analysis. An appendix is added as a last section.
\section{Well-posedness of the problem}
To start the well-posedness study for considered problem, we recall the energy space
$$ \mathcal{H}=H^1_0(\Omega)\times L^2(\Omega) $$
which is endowed with the inner product
$$ \left\langle (u_1,v_1),(u_2,v_2)\right\rangle_\mathcal{H}=\int_\Omega \left\{ \nabla u_1 \nabla u_2 + v_1 v_2 \right\}\;dx.$$
Denoting $v=u_t$, we can rewrite the problem \eqref{1} in $\mathcal{H}$ as the following abstract Cauchy problem 
\begin{equation}\label{8}
 \left\{\begin{array}{lll}
\frac{\partial}{\partial t}(u,v)+\mathcal{A}(u,v)=\mathcal{F}(u,v),\\
(u,v)(0)=(u_0,u_1),
\end{array}\right.
 \end{equation}
where the unbounded operator $\mathcal{A}:D(\mathcal{A}) \subset \mathcal{H} \rightarrow \mathcal{H}$ is given by
\begin{equation}\label{9}
\mathcal{A}(u,v)=\big(v,\mathop{\rm div}\big(\nabla u+ a(x)g(\nabla v)\big)+\eta(x)v\big),
\end{equation}
with domain 
\begin{equation}\label{10}
D(\mathcal{A})=\left\{ (u,v)\in \mathcal{H}: \; v\in H^1_0(\Omega),\;\; \mathop{\rm div}\big(\nabla u+ a(x)g(\nabla v)\big)\in L^2(\Omega)\right\},
\end{equation}
and $\mathcal{F}:\mathcal{H}\rightarrow \mathcal{H}$ is the operator 
\begin{equation}\label{a11}
\mathcal{F}(u,v)=(0,-f(u)).
\end{equation}
The first point is to prove that the operator $\mathcal{A}:D(\mathcal{A})\subset \mathcal{H}\rightarrow \mathcal{H}$ defined above generates a $C_0$-semigroup of contractions $e^{t\mathcal{A}}$ on the energy space $\mathcal{H}$ and the domain $D(\mathcal{A})$ is dense in $\mathcal{H}$.

Consider the operators $A$ and $B$ defined as follows
$$A:D(A) = H_0^1(\Omega)\subset H^{-1}(\Omega) \rightarrow H^{-1}(\Omega),$$
$$ u \mapsto Au=-\Delta u,$$
and
$$ B:H^1_0(\Omega) \rightarrow H^{-1}(\Omega),  $$
$$ v\mapsto Bv=-\mathop{\rm div} \big( a(x)g(\nabla v) \big)+\eta(x)v.  $$
It is easy to verify that $A$ is a maximal monotone operator. In other hand, we have
 $$\begin{array}{lll}
&&\left\langle Bu-Bv,u-v\right\rangle_{H^{-1}(\Omega),H_0^1(\Omega)}\\&=&\displaystyle\int_\Omega a(x)[g(\nabla u)-g(\nabla v)](\nabla u-\nabla v)\;dx+\int_\Omega\eta(x)|u-v|^2\;dx \\
&=&\displaystyle \sum_{i=1}^n \int_\Omega a(x)[g_i(\partial_{x_i}u)-g_i(\partial_{x_i}v)](\partial_{x_i}u-\partial_{x_i}v)+\int_\Omega\eta(x)|u-v|^2\;dx.
\end{array}
$$
Since the fact that $a$ and $\eta$ are non-negative functions and every component $g_i$ of $g$ is monotonic by hypothesis, we obtain that $B$ is monotone.

\medskip

We claim also that $B$ is hemicontinuous. In fact, consider $u$ and $v$ in $D(B)$ and $t_q \rightarrow 0$ when $q \rightarrow +\infty$, then, for all $w\in H_0^1(\Omega)$, we have 
\begin{equation}\label{12}
\begin{array}{lll}
&&\displaystyle\lim_{q\rightarrow \infty}\left\langle B(u+t_q v),w\right\rangle_{H^{-1}(\Omega),H_0^1(\Omega)}\\
&=&\displaystyle\lim_{q\rightarrow \infty}\int_\Omega a(x)g( \nabla u+t_q \nabla v) \nabla w\;dx +\int_\Omega \eta(x)(u+t_qv)w\;dx\\
&=&\displaystyle \sum_{i=1}^n \lim_{q\rightarrow \infty}  \int_\Omega a(x)g_i(\partial_{x_i}u+t_q\partial_{x_i}v)\partial_{x_i}w \;dx+\int_\Omega \eta(x)uw\;dx.
\end{array}
\end{equation}
Let $\psi_q^i(x)=a(x)g_i(\partial_{x_i}u+t_q\partial_{x_i}v)\partial_{x_i}w$. Taking in mind \eqref{6}, if $|\partial_{x_i}u+t_q\partial_{x_i}v|\geq 1 $, we have 
$$\begin{array}{lll}
|\psi_q^i(x)|&=& a(x)|g_i(\partial_{x_i}u(x)+t_q\partial_{x_i}v(x))||\partial_{x_i}w(x)|\\
&\leq& K \left\|a\right\|_\infty |\partial_{x_i}u(x)+t_q\partial_{x_i}v(x)||\partial_{x_i}w(x)|\\
&\leq& c_1 |\partial_{x_i}u(x)||\partial_{x_i}w(x)| + c_2 |\partial_{x_i}v(x)| |\partial_{x_i}w(x)|
\end{array}
$$
almost everywhere in $\Omega$, where $c_1$ and $c_2$ are two constants. The second constant result of the convergence of $(t_q)$.
Then, for all $1 \leq i \leq n, \psi_q^i \in L^1(\Omega)$ since $\partial_{x_i}u(x),\;\partial_{x_i}v(x),\;\partial_{x_i}w(x) \in L^2(\Omega).$

\medskip

In the case where $|\partial_{x_i}u+t_q\partial_{x_i}v|\leq 1 $, using the continuity of $g_i$, we obtain that $\psi_q^i \in L^1(\Omega)$ in the same way.\\
Using once more the continuity of $g_i$ and thanks to Lebesgue's dominated convergence theorem, we find 
$$ \lim_{q\rightarrow \infty} \int_\Omega a(x)g_i(\partial_{x_i}u(x)+t_q\partial_{x_i}v(x))\partial_{x_i}w(x) \;dx= \int_\Omega a(x)g_i(\partial_{x_i}u(x))\partial_{x_i}w(x) \;dx.$$
Moreover, we have
\begin{equation}\label{13}
\begin{array}{lll}
\displaystyle \sum_{i=1}^n \lim_{q\rightarrow \infty}  \int_\Omega a(x)g_i(\partial_{x_i}u+t_q\partial_{x_i}v)\partial_{x_i}w \;dx&=&
\displaystyle \displaystyle \sum_{i=1}^n\int_\Omega a(x)g_i(\partial_{x_i}u(x))\partial_{x_i}w(x) \;dx\\
&=& \displaystyle \int_\Omega a(x)g \big(\nabla u\big) \nabla w\;dx.
\end{array}
\end{equation}
Thus, \eqref{12} and \eqref{13} yields
$$ \lim_{q\rightarrow \infty}\left\langle B(u+t_q v),w\right\rangle_{H^{-1}(\Omega),H_0^1(\Omega)}=\left\langle Bu,w\right\rangle_{H^{-1}(\Omega),H_0^1(\Omega)}, $$
which gives that $B$ is hemicontinuous.

\medskip

Using \eqref{6}, we get 
\begin{equation}\label{14}
\begin{array}{lll}
\displaystyle \left\langle u+Au+Bu,u\right\rangle_{H^{-1}(\Omega),H_0^1(\Omega)}= \\
\displaystyle \left\langle u,u\right\rangle_{H^{-1}(\Omega),H_0^1(\Omega)}+\displaystyle \left\langle Au,u\right\rangle_{H^{-1}(\Omega),H_0^1(\Omega)}+\displaystyle \left\langle Bu,u\right\rangle_{H^{-1}(\Omega),H_0^1(\Omega)}\\
\geq  \displaystyle |u|^2+ \int_\Omega |\nabla u|^2\;dx \geq \left\|u\right\|_{H^1_0(\Omega)}^2,
\end{array}
\end{equation}
therefore, we have that $I+A+B$ is coercive.
Now, we rewrite the operator $\mathcal{A}:D(\mathcal{A})\subset \mathcal{H}\rightarrow \mathcal{H}$ as follows
$$ \mathcal{A} \left(\begin{array}{lll} u\\h \end{array}\right)=\left(\begin{array}{lll} -h\\Au+Bh \end{array}\right),$$
and 
$$ D(\mathcal{A})=\left\{ (u,h)\in \mathcal{H};\;\; h\in H^1_0(\Omega),\;\; Au+Bh\in L^2(\Omega) \right\}. $$
Then, we have to prove that $\mathcal{A}$ is maximal monotone in $\mathcal{H}$. Indeed, $\mathcal{A}$ is monotone, since
$$
\begin{array}{lll}
&&\displaystyle\left( \mathcal{A}\left[\begin{array}{lll} u_1\\h_1 \end{array}\right] - \mathcal{A}\left[\begin{array}{lll} u_2\\h_2 \end{array}\right], \left[\begin{array}{lll} u_1\\h_1 \end{array}\right]-\left[\begin{array}{lll} u_2\\h_2 \end{array}\right] \right)_\mathcal{H}\\
&=&\displaystyle\left(\left[\begin{array}{lll} -h_1+h_2 \\ A(u_1-u_2)+Bh_1-Bh_2\end{array}\right],\left[\begin{array}{lll} u_1-u_2 \\ h_1-h_2\end{array}\right]\right)_\mathcal{H}\\
&=& -(h_1-h_2,u_1-u_2)_{H^1_0(\Omega)}+(A(u_1-u_2),h_1-h_2)_{L^2(\Omega)}+(Bh_1-Bh_2,h_1-h_2)_{L^2(\Omega)}\\
&=&(Bh_1-Bh_2,h_1-h_2)_{L^2(\Omega)} \geq 0,
\end{array}$$
because $B$ is monotone.

\medskip

To show the maximality of $\mathcal{A}$, we are going to show that $R(I+\mathcal{A})=\mathcal{H}$. To do so, let $(u_0,h_0)\in \mathcal{H}$ and we look for an element $(u,h)\in D(\mathcal{A})$ such that 
\begin{equation}\label{15bis}
 (I+\mathcal{A}) \left[\begin{array}{lll} u\\h \end{array}\right]=\left[\begin{array}{lll} u_0\\h_0 \end{array}\right].
\end{equation}
Equivalently, we consider the following system
 $$
\left\{\begin{array}{lll}
      u-h=u_0,\\
			
      h+Au+Bh=h_0.
 \end{array}\right.
$$
Combining the above identities we deduce in the weak space
\begin{equation}\label{16}
h+Ah+Bh=h_0-Au_0.
\end{equation}
Therefore, it sufficient to show that $I+A+B$ is maximal monotone. For that, we have to prove that $R(I+A+B)=H^{-1}(\Omega)$. Since $B$ is monotone and hemicontinuous, $I+A$ is maximal monotone and $I+A+B$ is coercive, we obtain thanks to Corollary 1.3 in \cite{barbu} that $I+A+B$ is maximal monotone which implies that problem \eqref{16} admits a unique solution $h\in H^1_0(\Omega)$. Since $u=u_0+h$ and $Au+Bh=u_0-h$, we obtain that $u\in H^1_0(\Omega)$ and $Au+Bh \in L^2(\Omega)$.

\medskip

Thus, the problem \eqref{15bis} has a unique solution $(u,h)\in D(\mathcal{A})$, and, consequently, $\mathcal{A}$ is maximal monotone in $\mathcal{H}$.\\
Moreover, the nonlinear operator $\mathcal{F}:\mathcal{H}\rightarrow \mathcal{H}$ given in \eqref{a11} is a locally Lipschitz continuous operator. Indeed, given a bounded set $\mathcal{S}$ in $\mathcal{H}$ and $(u_1,u_2)$, $(v_1,v_2)$ in $\mathcal{S}.$ Then, using H\"older's inequality, the fact that $H^1$ is continuously embedded in $L^{2q}(\Omega)$ for $2q=\frac{22^*}{2^*-2(p-1)}\leq 2^*$ and by Assumption \eqref{3}, we get
$$ \begin{array}{lll}
\left\|\mathcal{F}(u_1,u_2)-\mathcal{F}(v_1,v_2)\right\|_\mathcal{H}^2 &\leq& \left\|f(u_1)-f(v_1)\right\|_{L^2(\Omega)}\\
&\leq& C \left( 1+\left\|u_1\right\|^{2(p-1)}_{L^{2^*}}+\left\|v_1\right\|^{2(p-1)}_{L^{2^*}} \right)\left\|u_1-v_1\right\|^2_{L^{2q}}\\
&\leq& C\left\|u_1-v_1\right\|^2_{H^1}\\
&\leq& C \left\|(u_1-u_2)-(v_1-v_2)\right\|^2_{\mathcal{H}}.
\end{array}$$
Therefore, thanks to Theorems 6.1.4 and 6.1.5 in Pazy's book \cite{LT4}, see also Theorem 7.2 in \cite{LT5}, the problem \eqref{8} has a unique mild solution on the interval $[0,T_{\max})$. This solution is given by 
$$ (u(t),v(t))=e^{t\mathcal{A}}(u_0,u_1)+\int_0^te^{(t-s)\mathcal{A}}\mathcal{F}(u(s),v(s))\;ds,\;\;\; \forall t\in [0,T_{\max}). $$
Furthermore, if $(u_0,u_1)\in D(\mathcal{A})$ the solution is regular. We claim that $T_{\max}=\infty$.

\medskip

Indeed, by a derivative of the energy functional defined in \eqref{7}, it follows that 
\begin{equation}\label{17}
 \frac{dE_u}{dt} (t)=-\int_\Omega a(x)g(\nabla \partial_tu)\nabla \partial_tu+\eta(x)|\partial_tu|^2\;dx, 
\end{equation}
which shows that the energy $E_u(t)$ is non-increasing and thus $E_u(t)\leq E_u(0)$ for all $t\in [0,T_{\max})$. Therefore, taking \eqref{5} in mind, we have 
$$ E_u(t)\geq \frac{1}{2}\int_\Omega |\partial_tu(x,t)|^2\;dx+ \frac{1}{2}\int_\Omega |\nabla u(x,t)|\;dx = \frac{1}{2} \left\|(u,u_t)\right\|_\mathcal{H}^2. $$ 
Therefore
$$ \frac{1}{2} \left\|(u,u_t)\right\|_\mathcal{H}^2 \leq E_u(t) \leq E_u(0),\;\;\; \forall t\in [0,T_{\max}),$$
and, thus, the local solutions cannot blow-up in finite time which gives that $T_{\max}=\infty.$
Finally, we have to give the following well-posedness theorem.
\begin{thm} \label{thm2.1}
Suppose that the conditions on the nonlinear term $f$ specified
in Assumption \ref{a1. 1} are satisfied in the case $1\leq p\leq \frac{n}{n-2}$ if $n\geq 3$, and the initial data $(u_0,u_1)$ is in the space $\mathcal{H}$. Then problem \eqref{8} admits a unique mild solution
$$ u\in C(0,+\infty;H_0^1(\Omega))\cap C(0,+\infty;L^2(\Omega)). $$
Further, if $(u_0,u_1)$ is in $D(\mathcal{A})$, then the solution is regular, i.e. $C(0,+\infty; H^2(\Omega) \cap H_0^1(\Omega))\cap C(0,+\infty;H^1_0(\Omega))$.
\end{thm}

\section{Uniform decay for problem \eqref{1}} 
This section is devoted to prove the uniform decay result for the system \eqref{1} under suitable growth hypothesis on the feedback and potential functions. Before stating our goal, we have to define some needed functions. We will adapt ideas firstly introduced in the literature by Lasiecka and Tataru \cite{LT21} for attractive forces and expanded by Cavalcanti et al. \cite{LT13} for repulsive forces (sources). Let consider $h_i$ be a concave, strictly increasing functions, with $h_i (0) = 0$, satisfying 
\begin{equation}\label{17'}
h_i(s_ih_i(s_i)\geq s_i^2+g_i^2(s_i),\;\mbox{ for } |s_i|<1,
\end{equation}
for $i=1,..,n$. Note that such function can be straightforwardly constructed, given the hypothesis on function $g_i$ given in \eqref{6}. We define, also, for $i=1,..,n$, the function $\hat{h}_i$ by $\hat{h}_i(\cdot)=h_i\left(\frac{\cdot}{meas(Q_T)}\right)$, where $Q_T=\Omega\times(0,T)$. With these later functions, let us take the function $\hat{h}$ defined as 
$$ \hat{h}(\cdot)=\sum_{i=1}^n\hat{h}_i(\cdot). $$
As $\hat{h}$ is monotone increasing, $\mu I + \hat{h}$ is invertible for all $\mu \geq 0$. For $L$ a positive constant, we
then set, respectively,
\begin{equation}\label{17''}
 z(x)=(\mu I+\hat{h})^{-1}(Lx),\;\;\; L:=(C(1+\left\|a\right\|_\infty) \, \mbox{meas}(Q_T))^{-1}, 
\end{equation}
where $C$ and $\mu$ are a positive constants that will be determined later.

\medskip

The function $z$ is easily seen to be positive, continuous, and strictly increasing with $z(0)=0.$
Finally, let
\begin{equation}\label{17'''}
\mathcal{R}=x-(I-z)^{-1}(x).
\end{equation}
 Let us consider $T_0>0$ associated to the geometric control condition introduced in Bardos, Lebeau and Rauch \cite{LT6} which means that every ray of the geometric optics enters $\omega:=\Omega/A$ in a time $T^*<T_0$. 

\medskip

Our goal is to show that the solutions of problem \eqref{1} decay uniformly to zero which reads as follows:
\begin{thm} \label{thm3.1}
Suppose that the Assumptions \ref{a1. 1}-\ref{a1.5} are fulfilled. Let $u$ be the mild solution of the problem \eqref{1} (according to Theorem \ref{thm2.1}). Then, the
following uniform decay rate estimate, for bounded initial data, holds
$$E_u(t)\leq S\left(\frac{t}{T_0}-1\right),\;\; \mbox{ for all } t>T_0,$$
where the contraction semigroup $S(t)$ is the solution of the differential equation 
$$ \frac{d}{dt}S(t)+\mathcal{R}(S(t))=0,\;\;\;S(0)=E(0), $$
where $\mathcal{R}$ is given in \eqref{17'''}. 
\end{thm}

The principal tool of the proof of this theorem consists to obtain an observability inequality for the system \eqref{1}. To find such a desired result some micro-local analysis techniques will be developed. 

\medskip
                               
Integrating the identity \eqref{17}, we get for all $0\leq t_1\leq t_2 <\infty$ the following one 
\begin{equation}\label{18}
E_u(t_2)-E_u(t_1)=-\int_{t_1}^{t_2}\int_\Omega\left\{ a(x)g(\nabla \partial_tu)\nabla \partial_tu+\eta(x)|\partial_tu|^2\right\}\;dxdt.
\end{equation}
Thus, we have to prove the following lemma.
\begin{lemma} \label{lemma3.2} For every $T\geq T_0$ and for $R>0$, there exists a real $c>0$ depending only in initial data such that the associated solution of \eqref{1} verifies that 
\begin{equation}\label{19}
 E_u(0)\leq c\left( \int_0^T\int_\Omega a(x)\left(|\nabla \partial_t u|^2 +|g(\nabla \partial_t u)|^2 \right)dxdt+ \int_0^T\int_\Omega\eta(x)|\partial_t u|^2\;dxdt \right), 
\end{equation}
where the energy-norm of the initial data is in $(0,2R)$.
\end{lemma}
\begin{proof}
The proof of this lemma will be conducted by arguing by contradiction. Suppose \eqref{19} is not true and hence there is $T\geq T_0$ and $R>0$ such that for all real $c>0$, there exists a solution $u_c$ for \eqref{1} satisfying $E_{u_c}(0)<R$ and such that $u_c$ violates the inequality \eqref{19}. \\
Particularly, for every $c=k\in \mathbb{N}$,  we obtain the existence of an initial datum $(u_{0,k},u_{1,k})$ such that $\left\|(u_{0,k},u_{1,k})\right\|_\mathcal{H}\leq 2R$, and whose corresponding solution $u_k$ satisfies
$$ E_{u_k}(0)\geq k\left( \int_0^T\int_\Omega a(x)\left(|\nabla \partial_t u_k|^2 +|g(\nabla \partial_t u_k)|^2 \right)dxdt+ \int_0^T\int_\Omega\eta(x)|\partial_t u_k|^2\;dxdt \right). $$
Further, we find a sequence of solution $(u_k)_{k\in \mathbb{N}}$ to problem \eqref{1} such that 
$$ \lim_{k\rightarrow \infty}\frac{E_{u_k}(0)}{\int_0^T\int_\Omega a(x)\left(|\nabla \partial_t u_k|^2 +|g(\nabla \partial_t u_k)|^2 \right)dxdt+ \int_0^T\int_\Omega\eta(x)|\partial_t u_k|^2\;dxdt}=+\infty, $$
which is equivalent to the following limit to zero
\begin{equation}\label{20}
\lim_{k\rightarrow \infty}\frac{\int_0^T\int_\Omega a(x)\left(|\nabla \partial_t u_k|^2 +|g(\nabla \partial_t u_k)|^2 \right)dxdt+ \int_0^T\int_\Omega\eta(x)|\partial_t u_k|^2\;dxdt}{E_{u_k}(0)}=0.
\end{equation}
Bearing in mind the boundedness of the sequence $(E_{u_k}(0))_{k\in \mathbb{N}}$, the convergence \eqref{20} gives 
\begin{equation}\label{21}
\lim_{k\rightarrow \infty}\int_0^T\int_\Omega a(x)\left(|\nabla \partial_t u_k|^2 +|g(\nabla \partial_t u_k)|^2 \right)dxdt+ \int_0^T\int_\Omega\eta(x)|\partial_t u_k|^2\;dxdt=0.
\end{equation}
It is worth mentioning, also, that there exists a subsequence of $(u_k)_{k\in \mathbb{N}}$, which will be denoted by
the same notation, such that 
\begin{equation}\label{22}
u_k \rightharpoonup u \mbox{ weakly}^\ast\; \mbox{ in } L^\infty(0,T;H^1_{0}(\Omega)),\;\; \mbox{ as } k\rightarrow +\infty,
\end{equation}
 \begin{equation}\label{23}
\partial_tu_k \rightharpoonup \partial_tu \mbox{ weakly}^\ast\; \mbox{ in } L^\infty(0,T;L^2(\Omega)),\;\; \mbox{ as } k\rightarrow +\infty,
\end{equation}
and, thanks to Aubin-Lions-Simon Theorem (see \cite{LT2}), we have
 \begin{equation}\label{24}
u_k \rightarrow u \mbox{ strongly }\; \mbox{ in } L^\infty(0,T;L^q(\Omega)),\;\; \mbox{ as } k\rightarrow +\infty,\;\mbox{ for all } q\in \left[2,\frac{2n}{n-2}\right).
\end{equation}
Accordingly, we can deduce that 
\begin{equation}\label{25}
f(u_k) \rightharpoonup f(u) \mbox{ weakly}^\ast\; \mbox{ in } L^2(0,T;L^2(\Omega)),\;\; \mbox{ as } k\rightarrow +\infty.
\end{equation}
Now, we are in position to distinguish two cases where $u\neq 0$ and $u= 0$. First of all, let us take the subset $\mathcal{V}$ of $A$ defined as 
$$ \mathcal{V}=\left\{ x\in A:\; d(x,y)>\frac{\epsilon}{2},\;\;y\in \partial A \right\}. $$
Case (i): $u\neq 0$.\\
We recall that for all $k\in \mathbb{N}$, $u_k$ solve the problem 
\begin{equation}\label{26}
 \left\{\begin{array}{lll}
       \partial_t^2u_k-\Delta u_k-\mathop{\rm div} \big( a(x)g(\nabla\partial_tu_k)\big)+\eta(x)\partial_tu_k+f(u_k)=0,\;\;&\mbox{in}& \;\; \Omega \times (0,\infty),\\
			
                    u_k =0, &\mbox{on}&\partial\Omega\times (0,\infty),\\
                       
 \displaystyle  u_k(x,0)=u_{0,k}(x),\;\;\;\partial_tu_k(x,0)=u_{1,k}(x), \;\;&\mbox{in}& \;\; \Omega.
                   
 \end{array}\right.
 \end{equation} 
Passing to the limit in \eqref{26} and exploring \eqref{21}-\eqref{25}, it follows  
\begin{equation}\label{27}
 \left\{\begin{array}{lll}
       \partial_t^2u-\Delta u+f(u)=0,\;\;&\mbox{in}& \;\; \Omega \times (0,\infty),\\
			
                    u =0, &\mbox{on}&\partial\Omega\times (0,\infty),\\
                       
 \partial_tu=0, \;\;&\mbox{in}& \;\; \Omega\backslash\mathcal{V}.
                   
 \end{array}\right.
 \end{equation} 
Denoting $w=\partial_tu$, the problem \eqref{27} implies
\begin{equation}\label{28}
 \left\{\begin{array}{lll}
       \partial_t^2w-\Delta w+V(x,t)w=0,\;\;&\mbox{in}& \;\; \Omega \times (0,\infty),\\
			
                    w =0, &\mbox{on}&\partial\Omega\times (0,\infty),\\
                       
 w=0, \;\;&\mbox{in}& \;\; \Omega\backslash\mathcal{V},
                   
 \end{array}\right.
 \end{equation} 
with $V(x,t)=f'(u(x,t)).$

Consequently, since Assumption \ref{a1.5}, we deduce that $w=\partial_tu=0$ and, thus, returning to \eqref{27} we deduce that $u=0$ which presents the desired contradiction as well.\\
Case (ii): $u=0$.\\
Now, we define
\begin{equation}\label{29}
\sigma_k:=\sqrt{E_{u_k}(0)},\;\;\;\mbox{and}\;\;\; v_k=\frac{u_k}{\sigma_k},
\end{equation}
and, thus, let us consider the following normalized problem
\begin{equation}\label{30}
 \left\{\begin{array}{lll}
      \displaystyle \partial_t^2v_k-\Delta v_k-\mathop{\rm div} \left( \frac{1}{\sigma_k}a(x)g(\nabla\partial_tu_k)\right)+\eta(x)\partial_tv_k+\frac{1}{\sigma_k}f(u_k)=0,\;&\mbox{in}& \; \Omega \times (0,\infty),\\			
                    v_k =0, &\mbox{on}&\partial\Omega\times (0,\infty),\\                      
 \displaystyle  v_k(x,0)=\frac{u_{0,k}}{\sigma_k},\;\;\;\partial_tv_k(x,0)=\frac{u_{1,k}}{\sigma_k},&\mbox{in}&  \Omega.
 \end{array}\right.
 \end{equation} 
We define the associated energy by 
\begin{equation}\label{301}
 E_{v_k}(t)=\frac{1}{2} \int_\Omega\left\{|\partial_tv_k(x,t)|^2+|\nabla v_k(x,t)|^2\right\}\;dx+\frac{1}{\alpha_k^2}\int_\Omega F(u_k(x,t))\;dx.
\end{equation}
It is easy to verify that 
\begin{equation}\label{31}
 E_{v_k}(t)=\frac{1}{\alpha_k^2} E_{u_k}(t),\;\;\forall t\geq 0.
\end{equation}
Note that, by the definition of $\sigma_k$, \eqref{31} gives in particular, for $t=0$, that $E_{v_k}(0)=1$. To conclude the contradiction, our strategy is to prove that $E_{v_k}(0) $ converges to zero. Indeed, first, taking \eqref{21} and \eqref{29} into account we infer
\begin{equation}\label{32}
\lim_{k\rightarrow \infty}\int_0^T\int_\Omega a(x)\left(|\nabla \partial_t v_k|^2 +\frac{1}{\sigma_k^2}|g(\nabla \partial_t u_k)|^2 \right)dxdt+ \int_0^T\int_\Omega\eta(x)|\partial_t v_k|^2\;dxdt=0.
\end{equation}
On the other hand, from the boundedness of $E_{v_k}(0) $, we also deduce that for an eventual subsequence of $(v_k)_{k\in \mathbb{N}}$, still denoted by the same way, satisfying
\begin{equation}\label{34}
v_k \rightharpoonup v \mbox{ weakly}^\ast\; \mbox{ in } L^\infty(0,T;H^1_{0}(\Omega)),\;\; \mbox{ as } k\rightarrow +\infty,
\end{equation}
 \begin{equation}\label{35}
\partial_tv_k \rightharpoonup \partial_tv \mbox{ weakly}^\ast\; \mbox{ in } L^\infty(0,T;L^2(\Omega)),\;\; \mbox{ as } k\rightarrow +\infty.
\end{equation}
From standard compactness arguments, we deduce, for an eventual subsequence, that will be denoted by the same notation, that
\begin{equation}\label{36}
v_k \rightarrow v \mbox{ strongly }\; \mbox{ in } L^\infty(0,T;L^q(\Omega)),\;\; \mbox{ as } k\rightarrow +\infty,\;\mbox{ for all } q\in \left[2,\frac{2n}{n-2}\right).
\end{equation}
Further, note that, for an eventual subsequence, $\sigma_k \rightarrow \sigma \in [0,\infty)$.

Subcase (a): $\sigma >0$. Recalling that we are in the case where $u=0$, using \eqref{25} and the fact that $f(0)=0 $, it follows
\begin{equation}\label{37}
\frac{1}{\sigma_k}f(u_k) \rightharpoonup 0, \mbox{ weakly in } L^2(0,T,L^2(\Omega)).
\end{equation}
Then, passing to the limit in \eqref{30} and using the convergences \eqref{32}-\eqref{37}, we infer that
\begin{equation}\label{38}
 \left\{\begin{array}{lll}
       \partial_t^2v-\Delta v=0,\;\;&\mbox{in}& \;\; \Omega \times (0,\infty),\\
			
                    v =0, &\mbox{on}&\partial\Omega\times (0,\infty),\\
                       
 \partial_tv=0, \;\;&\mbox{in}& \;\; \Omega\backslash\mathcal{V}.
                   
 \end{array}\right.
 \end{equation} 
Therefore, taking $y=\partial_tv,$ we obtain the following system 
\begin{equation}\label{39}
 \left\{\begin{array}{lll}
       \partial_t^2y-\Delta y=0,\;\;&\mbox{in}& \;\; \Omega \times (0,\infty),\\
			
                    y =0, &\mbox{on}&\partial\Omega\times (0,\infty),\\
                       
 y=0, \;\;&\mbox{in}& \;\; \Omega\backslash\mathcal{V}.
                   
 \end{array}\right.
 \end{equation} 
Once more, thanks to Assumption \ref{a1.5}, it follows that $y=\partial_tv=0$ which gives, since \eqref{38}, that $v=0$.

Subcase (b): $\sigma=0$. In this case, we are going to show that 
\begin{equation}\label{40}
\frac{1}{\sigma_k}f(u_k) \rightharpoonup f'(0)v, \mbox{ weakly in } L^2(0,T,L^2(\Omega)).
\end{equation}
First of all, the fact that $f\in C^2(\mathbb{R})$ and, thus, Taylor's formula allows us to get
\begin{equation}\label{41}
f(s) =f'(0)s+R(s), \mbox{ where } |R(s)|\leq c(|s|^2+|s|^{p+1}).
\end{equation}
This is obtain by using the condition \eqref{2} for the semi-linearity $f$.

\medskip

Thereafter, we have to write
\begin{equation}\label{42}
\frac{1}{\sigma_k}f(\sigma_k v_k)=f'(0)v_k+\frac{1}{\sigma_k}R(\sigma_kv_k).
\end{equation}
Now, to prove \eqref{40} it suffices to show that 
\begin{equation}\label{43}
\frac{1}{\sigma_k}R(\sigma_kv_k) \rightharpoonup 0\mbox{ weakly in } L^2(0,T,L^2(\Omega)).
\end{equation}
To do so, according \eqref{41} and \eqref{3}, we get 
\begin{equation}\label{44}
\begin{array}{lll}
\displaystyle\left\|\frac{1}{\sigma_k}R(\sigma_kv_k)\right\|^2_{L^2(0,T,L^2(\Omega))}&\leq& \displaystyle C\left\{\int_0^T\int_\Omega\left|\frac{f(\sigma_kv_k)}{\sigma_k}\right|^2\;dxdt+\int_0^T\int_\Omega |f'(0)v_k|^2\;dxdt\right\}\\\\
&\leq&\displaystyle C \left\{\frac{1}{\alpha_k^2} \int_0^T\int_\Omega\left(|\sigma_k v_k|^2+|\sigma_kv_k|^p\right)^2\;dxdt+\int_0^T\int_\Omega |v_k|^2\;dxdt \right\}\\\\
&\leq&\displaystyle C \left\{ \sigma_k^{2(p-1)}\left\|v_k\right\|^{2p}_{L^{2p}(0,T;L^{2p}(\Omega))} +\left\|v_k\right\|^{2}_{L^{2}(0,T;L^{2}(\Omega))} \right\}.
\end{array}
\end{equation}
Using the two facts that $H^1(\Omega)$ is continuously embedded in $L^{2p}(\Omega)$ and the boundedness of the sequence $(v_k)$, we can immediately conclude from \eqref{44} that $\left\{\frac{1}{\sigma_k}R(\sigma_kv_k)\right\}_k$ is bounded in the space $L^2(0,T,L^2(\Omega))$.

\medskip

Whereupon, this sequence admits a subsequence which converges in the associated space, that is, there exists $\delta\in L^2(0,T,L^2(\Omega))$ such that 
\begin{equation}\label{45}
\frac{1}{\sigma_k}R(\sigma_kv_k) \rightharpoonup \delta\mbox{ weakly in } L^2(0,T,L^2(\Omega)).
\end{equation}
Since \eqref{41}, we deduce that 
\begin{equation}\label{46}
\left|\frac{1}{\sigma_k}R(\sigma_kv_k)\right| \leq \sigma_k |v_k|^2+|\sigma_k|^p|v_k|^{p+1}\;\mbox{ in the case }\; 1\leq \frac{n}{n-2},
\end{equation}
which implies that
\begin{equation}\label{47}
\left\|\frac{1}{\sigma_k}R(\sigma_kv_k)\right\|_{L^1(0,T,L^1(\Omega))} \leq C \left(\sigma_k \left\|v_k\right\|^2_{L^2(0,T,L^2(\Omega))}+(\sigma_k)^p\left\|v_k\right\|^{p+1}_{L^{p+1}(0,T,L^{p+1}(\Omega))}\right).
\end{equation}
Therefore 
\begin{equation}\label{48}
\frac{1}{\sigma_k}R(\sigma_kv_k) \rightharpoonup 0 \mbox{ weakly in } L^1(0,T,L^1(\Omega)) \mbox{ as } k\rightarrow \infty.
\end{equation}
Combining \eqref{45} and \eqref{48} together with the embedding $L^2(0,T,L^2(\Omega))$ in $L^1(0,T,L^1(\Omega))$, it result the desired convergence in \eqref{43}. Consequently, \eqref{40} follows from \eqref{34}, \eqref{42} and \eqref{43}.
Next, observing all the convergences \eqref{32}-\eqref{35} and \eqref{40} and passing to the limit in \eqref{30}, we obtain in the distributional sense 
\begin{equation}\label{49}
 \left\{\begin{array}{lll}
       \partial_t^2v-\Delta v+f'(0)v=0,\;\;&\mbox{in}& \;\; \Omega \times (0,\infty),\\
			
                    v =0, &\mbox{on}&\partial\Omega\times (0,\infty),\\
                       
 \partial_tv=0, \;\;&\mbox{in}& \;\; \Omega\backslash\mathcal{V},
                   
 \end{array}\right.
 \end{equation} 
and, for $z=\partial_tv$, we infer 
\begin{equation}\label{50}
 \left\{\begin{array}{lll}
       \partial_t^2z-\Delta z+f'(0)z=0,\;\;&\mbox{in}& \;\; \Omega \times (0,\infty),\\
			
                    z =0, &\mbox{on}&\partial\Omega\times (0,\infty),\\
                       
 z=0, \;\;&\mbox{in}& \;\; \Omega\backslash\mathcal{V}.
                   
 \end{array}\right.
 \end{equation} 
As a consequence, via the Assumption \ref{a1.5} once again, it follows that $w=\partial_tu=0$. Returning to \eqref{27}, we immediately find that $v=0$.

\medskip

Thereby, in any of both cases $\sigma>0$ and $\sigma=0$, we find that $v=0$. As a consequence, $v=0$ in all convergences \eqref{34}-\eqref{36}.
Moreover, by virtue of \eqref{36}, \eqref{42} and \eqref{43} together with the nullity of $v$, it follows that
\begin{equation}\label{51}
\frac{1}{\sigma_k}f(u_k) \rightarrow 0, \mbox{ in } L^2(0,T,L^2(\Omega)).
\end{equation}
Recall that our objective is to show that $E_{v_k}(0)$ converges to zero, where $E_{v_k}(t)$ is defined by \eqref{301} as well. To do so, let us consider the D'Alembert operator given by 
$$ D:=\partial^2_t-\Delta. $$
By the first equation in \eqref{30}, we have
\begin{equation}\label{52}
Dv_k=\mathop{\rm div} \left( \frac{1}{\sigma_k}a(x)g(\nabla\partial_tu_k)\right)-\eta(x)\partial_tv_k-\frac{1}{\sigma_k}f(u_k),
\end{equation}
which gives, after applying the operator $\partial_t$, that 
\begin{equation}\label{53}
\partial_tDv_k=\partial_t\left(\mathop{\rm div} \left( \frac{1}{\sigma_k}a(x)g(\nabla\partial_tu_k)\right)-\eta(x)\partial_tv_k-\frac{1}{\sigma_k}f(u_k)\right).
\end{equation}
Our immediate task is to seek an appropriate limit to $\partial_tDv_k$, so, we have to determine the convergences of each term in the right hand side of \eqref{53}. Firstly, observing \eqref{51}, we deduce that 
\begin{equation}\label{54}
\frac{1}{\sigma_k}f(u_k) \rightarrow 0, \mbox{ strongly in } L^2(\Omega\times(0,T)) \mbox{ as } k\rightarrow \infty.
\end{equation}
Secondly, thanks to \eqref{32}, we have obviously that 
$$ \frac{1}{\sigma_k}a(x)g(\nabla \partial_tu_k)\rightarrow 0 \mbox{ strongly in } L^2(\Omega\times(0,T)) \mbox{ as } k\rightarrow \infty,$$
that is, 
\begin{equation}\label{55}
\mathop{\rm div} \left(\frac{1}{\sigma_k}a(x)g(\nabla \partial_tv_k)\right) \rightarrow 0, \mbox{ strongly in } H^{-1}_{loc}(\Omega\times(0,T)) \mbox{ as } k\rightarrow \infty.
\end{equation}
Also, by \eqref{32} once more, we deduce that 
\begin{equation}\label{56}
\eta(x)\partial_tv_k \rightarrow 0, \mbox{ strongly in } L^2(\Omega\times(0,T)) \mbox{ as } k\rightarrow \infty.
\end{equation}
Amalgamating \eqref{53} with \eqref{54}-\eqref{56}, it follows that
\begin{equation}\label{57}
D\partial_tv_k \rightarrow 0, \mbox{ strongly in } H^{-2}_{loc}(\Omega\times(0,T)) \mbox{ as } k\rightarrow \infty.
\end{equation}
Let $\chi$ be the microlocal defect measure associated with $(\partial_tv_k)_k$ in $L^2_{loc}(\Omega\times(0,T))$. In view of \eqref{57}, we deduce that the supp($\chi$) is contained in the characteristic set of the wave operator $\left\{\tau^2=\left\|\xi\right\|^2\right\}$, for that see Theorem \ref{thmA2} in Appendix or also Proposition 2.1 and Corollary 2.2 due to G\'erard \cite{LT8}. 

\medskip

Additionally, the existence of the dissipative effects yields that
\begin{equation}\label{58}
\partial_tv_k \rightarrow 0 \;\;\mbox{ in }\; (\Omega\backslash\mathcal{V}) \times (0,T).
\end{equation}
Therewith, the aim is to propagate the convergence of $\partial_tv_k$ to $0$ from $L^2((\Omega\backslash\mathcal{V}) \times (0,T))$ to the whole space $L^2(\Omega \times (0,T))$. It is worth mentioning that the convergence \eqref{57} is not regular enough to guarantee the desired propagation. This problem of regularity due to the following weak convergence:
\begin{equation}\label{59}
\partial_t\left(\mathop{\rm div} \left[ \frac{1}{\sigma_k}a(x)g(\nabla\partial_tu_k)\right]\right) \rightarrow 0, \mbox{ strongly in } H^{-2}_{loc}(\Omega\times(0,T)) \mbox{ as } k\rightarrow \infty.
\end{equation}
Whereupon, to obtain the propagation to the whole space, the presence of the frictional damping in the neighborhood $\mathcal{O}_\epsilon$ of the boundary of $\partial A $ plays a basic role. In fact, we have $\mathop{\rm div} \left( \frac{1}{\sigma_k}a(x)g(\nabla\partial_tu_k)\right)=0$ in $A\times (0,T)$ and, thus, taking in mind \eqref{53}, we find that
\begin{equation}\label{60}
D\partial_tv_k=\partial_t\left(-\eta(x)\partial_tv_k-\frac{1}{\sigma_k}f(u_k)\right)\rightarrow 0\mbox{ in } H^{-1}_{loc}(\mbox{int}(A)\times(0,T)) \mbox{ as } k\rightarrow \infty.
\end{equation}
In addition, thanks to general microlocal results given by Proposition \ref{prpA4} and Theorem \ref{thmA5} in Appendix, we can deduce that supp$(\chi)$ in $(\mbox{int}(A)\times(0,T))\times S^d $ presents a union of curves like
\begin{equation}\label{61}
t\in I\cap (0,\infty)\mapsto m_\pm (t)=\left( t,x(t),\frac{\pm 1}{\sqrt{1+|G(x)\dot x|^2}}, \frac{\mp G(x)\dot x}{\sqrt{1+|G(x)\dot x|^2}}\right),
 	\end{equation}
 	where $t\in I\mapsto x(t)\in \Omega$ is a geodesic associated to the metric.
	
\medskip

Therefore, The above convergence yields that $\chi$ propagates along the bicharacteristic flow of the operator 
$D(\cdot)$, which signifies, particularly, that if some point $\mu_0=(t_0,x_0,\tau_0,\xi_0)$ does not belong to the supp$(\chi)$, the whole bicharacteristic issued from $\mu_0$ is out of supp$(\chi)$. Further, since supp$\chi\subset \mathcal{V}\times (0,T)$, and the frictional damping effects in both sides of the boundary $\partial A$, we can propagate the kinetic energy from $(\mathcal{O}_{\epsilon/2}\cap A)\times (0,T)$ towards the set $\mathcal{V}\times (0,T)$. Indeed, consider $t_0\in (0,+\infty)$ and $x$ a geodesic of $G=I_d$ defined near $t_0$. Once the geodesics inside $\mathcal{V}$, enter necessarily in $\omega$, they again intersect the set $A\backslash\mathcal{V}$ , consequently, for any geodesic of the metric $G$, with $0\in I$ there exists $t>0$ such that $m_\pm(t)$ does not belong to supp$(\chi)$, so that $m_\pm(t_0)$ does not belong as well. Once the time $t_0$ and the geodesic $x$ were taken arbitrary, we conclude that supp$(\chi)$ is empty. Whereupon, we have
$$ \partial_tv_k \rightarrow 0, \mbox{ in } L^2(\omega\times (0,T)), $$
and considering \eqref{58}, we can deduce the following convergence in the the whole space
\begin{equation}\label{62}
\int_0^T\int_\Omega|\partial_tv_k(x,t)|^2\;dxdt\rightarrow 0,\mbox{ as } k\rightarrow \infty.
\end{equation}
The task ahead is to show that $E_{v_k}(0)\rightarrow 0$ as $k\rightarrow +\infty$. Let us consider the following cut-off function 
	\begin{align*}
		&\theta\in C^{\infty}(0,T),  \quad    0\leq \theta(t) \leq 1,  \quad   \theta(t)=1 \ \mbox{in} \ (\varepsilon,T-\varepsilon),\;\;\epsilon>0.
		\end{align*}
Multiplying  the first equation in \eqref{30} by $v_k \theta$ and integrating by parts, we get			
\begin{equation}\label{63}
\begin{array}{lll}
		&&\displaystyle -\int_0^T \theta(t)\int_\Omega|\partial_tv_k|^2\,dxdt - \int_0^T \theta'(t)\int_\Omega \partial_tv_k v_k\,dxdt
		+\int_0^T \theta(t) \int_\Omega |\nabla v_k|^2 \,dxdt\\
		&&\displaystyle+ \frac{1}{\sigma_k}\int_0^T\theta(t)\int_\Omega f(\sigma_kv_k)v_k\,dxdt 
		+ \frac{1}{\sigma_k}\int_0^T\theta(t)\int_\Omega a(x)g(\nabla\partial_tu_k)\nabla v_k\,dxdt  \\
		 &&\displaystyle+ \int_0^T\theta(t)\int_{\Omega}\eta(x)\partial_tv_kv_k\,dx dt  =0.
		\end{array}
		\end{equation}
Employing the convergences \eqref{32}-\eqref{36} and \eqref{62}, and, having $v=0$, the identity \eqref{63} implies that 
\begin{equation}\label{64}
\lim_{k\rightarrow+\infty}\int_0^{T}\theta(t)\int_\Omega|\nabla v_k|^2+\frac{1}{\sigma_k}f(\sigma_k u_k)v_k\;dxdt =0,
\end{equation}
that is,
\begin{equation}\label{65}
\lim_{k\rightarrow+\infty}\int_\epsilon^{T-\epsilon}\int_\Omega|\nabla v_k|^2+\frac{1}{\sigma_k}f(\sigma_k u_k)v_k\;dxdt =0.
\end{equation}
Furthermore, by the definition of functions $f$ and $F$ and the property as well, the convergence in \eqref{65} yields 
\begin{equation}\label{66}
\lim_{k\rightarrow+\infty}\frac{1}{\sigma_k^2}\int_\epsilon^{T-\epsilon}\int_\Omega F( u_k)\;dxdt =0.
\end{equation}
Thereafter, combining \eqref{62}, \eqref{65} and \eqref{66}, it follows that 
 \begin{equation}\label{67}
 \int_\epsilon^{T-\epsilon} E_{v_k}(t)\,dt \rightarrow 0\mbox{ as } k\rightarrow +\infty.
 \end{equation}
Thus, by the decrease of the energy, it result 
 \begin{equation}\label{68}
 (T-2\epsilon)E_{v_k}(T-\epsilon) \rightarrow 0,\mbox{ as } k\rightarrow +\infty. 
 \end{equation}
The energy identity implies that 
$$ E_{v_k}(T-\epsilon)-E_{v_k}(\epsilon)=-\int_\epsilon^{T-\epsilon}\int_\Omega\frac{1}{\sigma_k} a(x)g(\nabla \partial_tu_k)\nabla \partial_tv_k+\eta(x)|\partial_tv_k|^2\;dxdt. $$ 
This combining with \eqref{32} and \eqref{68} and the arbitrariness of $\epsilon$, we achieve the desired contradiction
$$ E_{v_k}(0)\rightarrow 0,\mbox{ as } k\rightarrow +\infty, $$
and the proof of Lemma \ref{lemma3.2} is, thereby, finished.
\end{proof}
\begin{proof}\textit{of Theorem \ref{thm3.1}} In order to obtain the decay rate of energy to the problem \eqref{1}, it sufficient to apply the Lemma \ref{lemma3.3}. Indeed, Since \eqref{18}, \eqref{19} we have
 \begin{equation}\label{69}
\begin{array}{lll}
 E_u(T) &\leq&\displaystyle C \sum_{i=1}^n \int_0^T\int_\Omega a(x)\left(|\partial_{x_i} \partial_t u|^2 +|g_i(\partial_{x_i} \partial_t u)|^2 \right)dxdt+ C\int_0^T\int_\Omega\eta(x)|\partial_t u|^2\;dxdt,\\
&\leq& \displaystyle C \sum_{i=1}^n \int_{\Sigma_A^i\cup \Sigma_B^i} a(x)\left(|\partial_{x_i} \partial_t u|^2 +|g_i(\partial_{x_i} \partial_t u)|^2 \right)dxdt+ C\int_0^T\int_\Omega\eta(x)|\partial_t u|^2\;dxdt,
\end{array}
\end{equation}
where $\Sigma_A^i=\left\{ (t,x)\in Q_T;\;\;|\partial_{x_i} \partial_t u|>1 \mbox{ a.e.} \right\}$, $\Sigma_B^i=Q_T\backslash\Sigma_A^i$ and $Q_T$ is already defined in the begin of this section.

\medskip

On other hand, taking in mind the hypothesis \eqref{6}, we find 
 \begin{equation}\label{70}
\int_{\Sigma_A^i} a(x)\left(|\partial_{x_i} \partial_t u|^2 +|g_i(\partial_{x_i} \partial_t u)|^2 \right)\;dxdt\leq (m^{-1}+M)\int_{\Sigma_A^i}a(x)g_i(\partial_{x_i} \partial_t u)\partial_{x_i} \partial_t u \;dx dt,
 \end{equation}
and, since \eqref{17'}, the facts that, for $i=1,..,n$, $h_i$ is a concave, increasing function and that $\frac{a(x)}{1+\left\|a\right\|_\infty}\leq a(x)$, it follows 
 \begin{equation}\label{71}
\begin{array}{lll}
\displaystyle \int_{\Sigma_B^i} a(x)\left(|\partial_{x_i} \partial_t u|^2 +|g_i(\partial_{x_i} \partial_t u)|^2 \right)dxdt\\
\leq \displaystyle\int_{\Sigma_B^i} a(x) h_i\left(g_i(\partial_{x_i} \partial_t u)\partial_{x_i} \partial_t u\right)\;dxdt\\
\leq \displaystyle\int_{\Sigma_B^i}(1+\left\|a\right\|_\infty)  h_i\left(\frac{a(x)}{1+\left\|a\right\|_\infty}g_i(\partial_{x_i} \partial_t u)\partial_{x_i} \partial_t u\right)\;dxdt\\
\leq \displaystyle\int_{\Sigma_B^i}(1+\left\|a\right\|_\infty)  h_i\left(a(x)g_i(\partial_{x_i} \partial_t u)\partial_{x_i} \partial_t u\right)\;dxdt.
\end{array}
\end{equation}
Thereafter, using Jensen's inequality in \eqref{71}, we have
 \begin{equation}\label{72}
\begin{array}{lll}
\displaystyle \int_{\Sigma_B^i} a(x)\left(|\partial_{x_i} \partial_t u|^2 +|g_i(\partial_{x_i} \partial_t u)|^2 \right)dxdt\\
\leq \displaystyle(1+\left\|a\right\|_\infty) \, \mbox{meas}(Q_T)h_i\left(\frac{1}{\mbox{meas}(Q_T)}\int_{Q_T} a(x)g_i(\partial_{x_i} \partial_t u)\partial_{x_i} \partial_t u\;dxdt\right)\\
\leq \displaystyle (1+\left\|a\right\|_\infty) \, \mbox{meas}(Q_T)\hat{h}_i\left(\int_{Q_T} a(x)g_i(\partial_{x_i} \partial_t u)\partial_{x_i} \partial_t u\;dxdt\right).
\end{array}
\end{equation}
Therewith, substituting \eqref{70} and \eqref{72} in \eqref{69}, we can deduce  
\begin{equation}\label{73}
\begin{array}{lll}
 E_u(T) &\leq&\displaystyle C  (m^{-1}+M)\int_{Q_T}a(x)g(\nabla \partial_t u)\nabla \partial_t u \;dx dt+C\int_{Q_T}\eta(x)|\partial_t u|^2\;dxdt\\
&+&\displaystyle C(1+\left\|a\right\|_\infty) \, \mbox{meas}(Q_T)\sum_{i=1}^n\hat{h}_i\left(\int_{Q_T} a(x)g_i(\partial_{x_i} \partial_t u)\partial_{x_i} \partial_t u\;dxdt\right)\\
&\leq& \displaystyle C  (m^{-1}+M)\int_{Q_T}a(x)g(\nabla \partial_t u)\nabla \partial_t u \;dx dt+\int_{Q_T}\eta(x)|\partial_t u|^2\;dxdt\\
&+&\displaystyle C(1+\left\|a\right\|_\infty)\mbox{meas}(Q_T)\hat{h}\left(\int_{Q_T} a(x)g(\nabla \partial_t u)\nabla \partial_t u\;dxdt\right).
\end{array}
 \end{equation}
Therefore, since the increasing of $\hat{h}$, it follows that 
\begin{equation}\label{74}
\begin{array}{lll}
LE_u(T)&\leq& \displaystyle \mu \left\{\int_{Q_T}a(x)g(\nabla \partial_t u)\nabla \partial_t u \;dx dt+\int_{Q_T}\eta(x)|\partial_t u|^2\;dxdt\right\}\\
&+&\displaystyle\hat{h}\left(\int_{Q_T} a(x)g(\nabla \partial_t u)\nabla \partial_t u\;dxdt+\int_{Q_T}\eta(x)|\partial_t u|^2\;dxdt \right)\\
&\leq&\displaystyle \left(\mu I+\hat{h}\right)\left(\int_{Q_T} a(x)g(\nabla \partial_t u)\nabla \partial_t u\;dxdt+\int_{Q_T}\eta(x)|\partial_t u|^2\;dxdt \right),
\end{array}
\end{equation}
where $$ L= \frac{1}{C(1+\left\|a\right\|_\infty)\mbox{meas}(Q_T) }\;\; \mbox{ and }\;\; \mu=\frac{1+m^{-1}+M}{(1+\left\|a\right\|_\infty)\mbox{meas}(Q_T)}.$$\\
Furthermore, taking the function $z$ in \eqref{17''}, we obtain 
\begin{equation}\label{75}
z(E_u(T))\leq \int_{Q_T} a(x)g(\nabla \partial_t u)\nabla \partial_t u\;dxdt+\int_{Q_T}\eta(x)|\partial_t u|^2\;dxdt,
\end{equation}
which implies, thanks to \eqref{18}, that
\begin{equation}\label{76}
z(E_u(T))+E_u(T)\leq E(0).
\end{equation}
Whereupon, we replace $T$ (resp. $0$) in \eqref{76} with $(m+1)T$ (resp. $mT$) to obtain
\begin{equation}\label{76bis}
z(E_u((m+1)T))+E_u((m+1)T)\leq E(mT),\;\;\; \mbox{for }\; m=0,1,2,...
\end{equation}
Now, to deduce the proof of Theorem \ref{thm3.1}, it sufficient to use the following lemma due to Lasiecka and Tataru \cite{LT21}.
\begin{lemma} \label{lemma3.3}
Let $z$ be a positive, increasing function such that $z(0) = 0$. Since $z$ is increasing,
we can define an increasing function $\mathcal{R}$, $\mathcal{R}=x -(I + z)^{-1}(x)$. Consider a sequence $s_m$ of positive numbers which satisfies 
$$ s_{m+1}+ z(s_{m+1})\leq s_m. $$
Then, $s_m\leq S(m)$ where $S(t)$ is a solution of the differential equation 
$$ \frac{d}{dt}S(t)+\mathcal{R}S(t)=0,\;\;\;S(0)=s_0.$$
Moreover, if $z(x)>0$ for $x>0$, then $\displaystyle\lim_{t\rightarrow +\infty}S(t)=0$.
\end{lemma}
Applying Lemma \ref{lemma3.3} with $s_m = E(mT)$, thus results in
\begin{equation}\label{77}
E_u(mT) \leq S(m), \;\;\mbox{ for }\; m = 0, 1, 2 . . .
\end{equation}
Lastly, using the dissipativity of $E(t)$ and S(t), we have for $t = mT + \tau$ and $0 \leq \tau \leq T$, 
\begin{equation}\label{78}
E_u(t)\leq E_u(mT) \leq S(m)=S\left(\frac{t-\tau}{T}\right)\leq S\left(\frac{t}{T}-1\right), \;\;\mbox{ for }\; t>T.
\end{equation}
The proof of decay for $E_u$, and thus of the Theorem \ref{thm3.1}, is now complete.
\end{proof}

\section{Appendix: Microlocal Analysis Background}

For the readers comprehension, we will announce some results which can be found in Burq and G\'erard \cite{LT12.}  and in G\'erard \cite{LT8} and were used in the proof of the uniform stability result.

\begin{thm}\label{thmA1}
Let $\{u_n\}_{n\in \mathbb{N}}$ be  a bounded sequence in $L_{loc}^2(\mathcal{O})$ such that it converges weakly to zero in $L_{loc}^2(\mathcal{O})$. Then, there exists a subsequence $\{u_{\varphi(n)}\}$ and a positive Radon measure $\mu$ on $T^1\mathcal{O}:=\mathcal{O}\times S^{d-1}$ such that for all pseudo-differential operator $A$ of order $0$ on $\Omega$ which admits a principal symbol $\sigma_0(A)$ and for all $\chi \in C_0^\infty(\mathcal{O})$ such that $\chi \sigma_0(A)=\sigma_0(A)$, one has
\begin{eqnarray}\label{M6}
\left(A(\chi u_{\varphi(n)}), \chi u_{\varphi_n}\right)_{L^2}\underset{n\rightarrow +\infty}\longrightarrow \int_{\mathcal{O} \times S^{n-1}}\sigma_0(A)(x,\xi)\,d\mu(x,\xi).
\end{eqnarray}
\end{thm}

\begin{Definition}
Under the circumstances of Theorem \ref{thmA1} $\mu$ is called the  microlocal defect measure of the sequence $\{u_{\varphi(n)}\}_{n\in \mathbb{N}}$.
\end{Definition}

\begin{remark}
Theorem \ref{thmA1} assures that for all bounded sequence $\{u_n\}_{n\in \mathbb{N}}$ of $L_{loc}^2(\mathcal{O})$ which converges weakly to zero, the existence of a subsequence admitting a microlocal defect measure. We observe that from (\ref{M6}) in the particular case when $A=f\in C_0^\infty(\mathcal{O})$, it follows that
\begin{eqnarray}\label{M22}
\int_\Omega f(x) |u_{\varphi(n)}(x)|^2\,dx \rightarrow \int_{\mathcal{O} \times S^{d-1}}f(x)\,d\mu(x,\xi),
\end{eqnarray}
so that $u_{\varphi(n)}$ converges to $0$ strongly  if and only if $\mu=0$.
\end{remark}

The second important result reads as follows.

\begin{thm}\label{thmA2}
Let $P$ be a differential operator of order $m$ on $\Omega$ and let $\{u_n\}$ a bounded sequence of $L_{loc}^2(\mathcal{O})$ which converges weakly to $0$ and admits a m.d.m. $\mu$. The following statement are equivalents:
\begin{eqnarray*}
&(i)&Pu_n \underset{n \rightarrow +\infty}\longrightarrow 0 \hbox{ strongly in }H_{loc}^{-m}(\mathcal{O})~(m>0). \\
&(ii)& \hbox{supp} (\mu) \subset \{(x,\xi)\in \mathcal{O} \times S^{n-1}: \sigma_m(P)(x,\xi)=0\}.
\end{eqnarray*}
\end{thm}

\begin{thm}\label{thmA3}
Let $P$ be a differential operator of order $m$ on $\Omega$, verifying $P^\ast=P$, and let $\{u_n\}$ be a bounded sequence in $L_{loc}^2(\mathcal{O})$ which converges weakly to $0$ and it admits a m.d.m. $\mu$. Let us assume that $P u_n \underset{n\rightarrow +\infty}\longrightarrow 0$ strongly in $H_{loc}^{1-m}(\mathcal{O})$. Then, for all function $a\in C^\infty(\mathcal{O} \times \mathbb{R}^n \backslash\{0\} )$ homogeneous of degree $1-m$ in the second variable and with compact support in the first one,
\begin{eqnarray}\label{M26}
\int_{\Omega \times s^{n-1}}\{a,p\}(x,\xi)\,d\mu(x,\xi)=0.
\end{eqnarray}
\end{thm}

Next, Let us consider the wave operator in the following general setting:
\begin{eqnarray*}
	\rho(x) \partial_t^2 - \sum_{i,j=1}^n \partial_{x_i}\left[K(x) \partial_{x_j}\right].
\end{eqnarray*}

Using the notation $D_j=\frac{1}{i}\partial_j$ we can write
\begin{eqnarray*}
	P(t, x, D_t, D_{x})= - \rho(x) D_t^2 +   \sum_{i,j=1}^n D_{x_i}[K(x)D_{x_j}],\quad D_x=(D_{x_1},\cdots, D_{x_n}),
\end{eqnarray*}
whose principal symbol $p(t,x,\tau,\xi)$ is given by
\begin{eqnarray}\label{5.5}
p(t,x,\tau,\xi)=-\rho(x) \tau^2 + K(x) \,\xi\cdot \xi,\quad\xi=(\xi_1,\cdots,\xi_n),
\end{eqnarray}
where $t\in \mathbb{R}$,~ $x\in \Omega\subset \mathbb{R}^d$,~ $(\tau,\xi)\in \mathbb{R}\times \mathbb{R}^d$.

\begin{prp}\label{prpA4}
	Unless a change of variables, the bicharacteristics of \eqref{5.5} are curves of the form
	\begin{eqnarray*}
		t \mapsto \left(t, x(t), \tau, -\tau\left(\frac{K(x(t))}{\rho(x(t))}\right)^{-1}\dot{x}(t) \right),
	\end{eqnarray*}
	where $t \mapsto x(t)$ is a geodesic of the metric $G=\left(K/\rho\right)^{-1}$ on $\Omega$, parameterized by the curvilinear abscissa.
\end{prp}

The main result is the following:

 \begin{thm}\label{thmA5}
 Let $P$ be a self-adjoint differential operator of order $m$ on $\mathcal{O}$ which admits a principal symbol $p$. Let $\{u_n\}_n$ be a bounded sequence in $L_{loc}^2(\mathcal{O})$ which converges weakly to zero, with a microlocal defect measure $\mu$. Let us assume that $P u_n$ converges to $0$ in $H_{loc}^{-(m-1)}(\mathcal{O})$. Then the support of $\mu$, $\hbox{supp}(\mu)$, is a union of curves like $s\in I \mapsto \left(x(s), \frac{\xi(s)}{|\xi(s)|}\right)$, where $s\in I \mapsto (x(s),\xi(s))$ is a bicharacteristic of $p$.
 \end{thm}

\end{document}